%% file: article_plain.tex
\documentclass[a4paper, 11pt, english]{article}
\usepackage{babel}
\usepackage[latin1]{inputenc}
\usepackage[T1]{fontenc}
\usepackage{amsmath}

 \usepackage{algorithmic}
\usepackage{algorithm}
\usepackage{graphicx}
\usepackage{amssymb}
\usepackage{amsthm}
\usepackage{subcaption}
\graphicspath{{figures/}}

\newtheorem{thm}{Theorem}

\newtheorem{lem}[thm]{Lemma}

\newtheorem{remark}[thm]{Remark}
\newtheorem{proposition}[thm]{Proposition}

\newcommand{\ee}{{\rm e}\hspace{1pt}}
\newcommand{\ii}{\text{i}\hspace{1pt}}
\newcommand{\dd}{\hspace{1pt}{\rm d}\hspace{0.5pt}}

\input AmsLtxMacros.tex

\def\Rtn{\R^{2n}}

\begin{document}

\title{Krylov integrators for Hamiltonian systems}





\author{Timo Eirola and Antti Koskela \vspace{5mm} \\ In memory of Timo  Eirola  (1951--2016)  
}


\date{}

\maketitle

\begin{abstract}We consider Arnoldi like processes to obtain symplectic
subspaces for Hamiltonian systems. Large systems are locally approximated 
by ones living in low dimensional subspaces; we especially consider Krylov 
subspaces and some extensions. This will be utilized
in two ways: solve numerically local small dimensional systems or in a
given numerical, e.g. exponential, integrator, use the subspace for
approximations of necessary functions. In the former case one can expect an
excellent energy preservation. For the latter this is so for linear systems.
For some second order exponential integrators we consider these two approaches
are shown to be equivalent.
In numerical experiments with nonlinear Hamiltonian
problems their behaviour seems promising.

\end{abstract}

\section{Introduction} 

Symplectic methods have shown to be very effective in long time
integration of Hamiltonian systems (see \cite{Hairer_Lubich_Wanner}). 
Many of them are implicit and necessitate the solution of systems of equations. If the differential
equation system is large and sparse, a natural approach is to use
Krylov subspace techniques to approximate solution of the algebraic
equations.

A related approach is to use Krylov approximations of the matrix exponential
in the so-called exponential integrators (see \cite{Hochbruck_Lubich_Selhofer}).
This has turned out to be a superior technique for
many large systems of ordinary differential equations.

Krylov subspace techniques can be viewed as local low dimensional approximations
of the large system. For Hamiltonian systems the standard Arnoldi type iterations
produce low dimensional systems that are no longer Hamiltonian.
In this paper special attention is paid to produce subspaces with symplectic
bases. Also the time symmetry of the Hamiltonian systems taken into account
when producing the bases.

This is an extended version of the slides by Eirola, presented at a Workshop on 
Exponential Integrators in Innsbruck in 2004 (see~\cite{Eirola}). Part of the material
was introduced in the master's thesis of the second author~\cite{Koskela} which was supervised by Eirola.
The original ideas of Eirola came from considering linear 
Hamiltonian systems in $\,\R^{2n}\,$ as $\,\R\,$--linear\footnote{
These are of the form: $\,\z\mapsto\M\z+\M_{\scriptscriptstyle\#}\overline\z\,$ for
$\,\M,\M_{\scriptscriptstyle\#}\in\Cnxn\,$,} systems in $\,\Cn\,$
(see~\cite{Eirola2}). The slides~\cite{Eirola} were using that language, but the present 
version is written in a more standard form.

\section{Hamiltonian systems}
Given a smooth function $\,H\;:\;\Rtn\to\R\,$, consider the Hamiltonian system
\begin{equation}
  \label{eq:Ham_sys}
 \x'(t)=\J^{-1}\,\nabla H(\x(t))\ ,\qquad \x(0)=\x_0\ ,
\end{equation}
where $\J=\smat{0&I\\-I&0}$. A matrix 
$\,\A\in\R^{2n\times 2n}\,$ is called Hamiltonian, if $\,\A^T=\mat{JAJ}\,$, and
symplectic, if $\,\A^T\J\A=\J\,$. The Jacobian of
$\,\nabla H(\x)\,$ is a symmetric matrix at every point. Thus
$\,D_{\x}\,\J^{-1}\,\nabla H(\x)\,$ is a Hamiltonian matrix. 
 
We assume that \eqref{eq:Ham_sys} has a unique solution and write it
$\,\x(t)=\bphi^t(\x_0)\,$. Then
\begin{itemize}\itemsep3pt
\item Energy is preserved: $\,H(\x(t))\,$ is constant in $\,t\,$. 
\item For every $\,t\,$ the mapping $\,\bphi^t\;:\;\Rtn\to\Rtn\,$ 
is symplectic (or canonical), that is, its derivative is a symplectic 
matrix at every point. 
\item The mapping $\,\bphi^t\,$ is time symmetric, i.e. $\,\bphi^{-t}(\x(t))=\x_0\,$ for every $\,t\,$.
\end{itemize}

Symplectic integrators produce symplectic one step maps
for Hamiltonian systems (see \cite{Hairer_Lubich_Wanner}).
For example, the implicit midpoint rule
\[\x_{j+1}=\x_j+h\,\J^{-1}\,\nabla H((\x_j+\x_{j+1})/2)\]
is such. For linear systems, i.e., when $\,H\,$ is of the form
$\, H(\x)=\thalf\,\x^T\S\,\x+\vec c^T\x\,$, the energy is also
preserved in the numerical solution with this and many other
symplectic methods. One step methods are called symmetric the map given
by the integrator is time symmetric, i.e. chaning $\,h\,$ 
to $\,-h\,$ is equivalent to switching $\,\x_j\,$ and  $\,\x_{j+1}\,$.
The implicit midpoint rule, for example, is symmetric.

For large systems implicit methods may become expensive. 
In this paper we will consider several low dimensional Hamiltonian 
approximations and the use of implicit methods or exponential
integrators for these.

\section{Symplectic subspaces and low dimensional approximations}

Recall some basic definitions and properties in $\,\Rtn\,$.
Denote the nondegenerate skew-symmetric bilinear form $\,\omega(\x,\y):=\x^T\J\y$.
A subspace $\,V\,$ is \emph{isotropic}, if $\,\omega(\x,\y)=0\,$ for all
$\,\x,\y\in V\,$, and a subspace $\,W\,$ is \emph{symplectic}, if for 
every nonzero $\,\x\in W\,$ there exists $\,\y\in W\,$ such that 
$\,\omega(\x,\y)\not=0\,$. Then the dimension of $\,W\,$ is even.
A basis $\,\e_1,\dots,\e_k,\f_1,\dots,\f_k\in W\,$ is called symplectic, or a \it Darboux basis\rm, if 
for all $\,i,j=1,\dots,k\,$ holds $\,\omega(\e_i,\e_j)=\omega(\f_i,\f_j)=0\,$ and 
$\,\omega(\e_i,\f_j)=\delta_{i,j}\,$.

If $\,V\,$ is an isotropic subspace with an orthonormal basis 
$\,\e_1,\dots,\e_k\,$, then $\,\e_1,\dots,\e_k\,$ are also $\,\omega\,$--orthogonal
and $\,W=V\oplus\J V\,$ is a symplectic subspace and 
$\,\e_1,\dots,\e_k,\,$ $\J^{-1}\e_1,\dots,\J^{-1}\e_k\,$
is a symplectic basis of $\,W\,$. 

We call also a matrix $\,\U\in\R^{2n\times2k}\,$ symplectic, if
(pointing out the dimensions) $\,\U^T\J_n\U=\J_k\,$. 
Then $\,\U^\dagger=\J_k^{-1}\U^T\J_n\,$
is a left inverse of $\,\U\,$ if and only if $\U$ is symplectic.

We will consider local approximations of the Hamiltonian system
\[\x'(t)=\f(\x(t))=\J^{-1}\nabla H(\x(t))\ .\]
Assume that at a point $\,\x_0\in\Rtn\,$ we are given a symplectic 
matrix $\,\U\in\R^{2n\times2k}\,$.
Consider the Hamiltonian system in $\,\R^{2k}\,$ corresponding to the function
$\,\eta(\bxi)=H(\x_0+\U\bxi)\,$. Then we get
\[\bxi'=\J^{-1}\,\nabla\eta(\bxi)=\J^{-1}\,\U^T\nabla H(\x_0+\U\bxi)\ ,\]
which is Hamiltonian in $\,\R^{2k}\,$. Set $\,\U^\dagger=\J^{-1}\U^T\J\,$. Then
\begin{equation}
  \label{eq:Ham_sys_s}
\bxi'(t)=\U^\dagger\f(\x_0+\U\bxi(t))\ .
\end{equation}

One strategy is to solve \eqref{eq:Ham_sys_s} numerically 
from $\,\bxi_0=0\,$ up to $\,\bxi_1\approx\bxi(t_1)\,$ 
and set $\,\x_1=\x_0+\U\bxi_1\,$. 
Clearly, if we use an energy preserving scheme for the system \eqref{eq:Ham_sys_s}, 
we will conserve the energy of the large system too, i.e. 
$\,H(\x_1)=H(\x_0)\,$. 

Note that if the sets of constant energy of the original system are bounded, 
then they are such for the small dimensional approximations too. 
This implies that the approximations inherit stability of equilibria
in a natural way.

In case that at $\,\x_0\,$ we are given a matrix $\,\U\,$ with orthonormal columns
we set $\,\U^\dagger=\U^T\,$ in \eqref{eq:Ham_sys_s}. Then the system is not 
necessarily Hamiltonian.

We will consider also another strategy which is, instead of solving low-dimesional 
systems, we approximate suitable functions of a numerical method
in the low dimensional space $\,\R^{2k}\,$. As we will see, for the exponential integrators we consider
these two approaches are equivalent.

The idea of approximating a Hamiltonian system by another of smaller dimension
is not new. See, for example the discussion in \cite{Lubich}.
A novelty here is to use local (later Krylov) approximations.
 
If $\,\U\,$ is symplectic and does not depend on $\,\x_0\,$, then it is not difficult 
to prove, for example, that using the
implicit midpoint rule for \eqref{eq:Ham_sys_s} induces a map 
$\,\bpsi\;:\;\x_0\to\x_1\,$ that is symplectic in $\,R(\U)\,$,
that is
\[\omega(\D\bpsi(\x_0)\vec d\,,\,\D\bpsi(\x_0)\tilde{\vec d})
=\omega(\vec d,\tilde{\vec d})\itext{for all}\vec d,
\tilde{\vec d}\in R(\U)\ .\]
But in order to get efficient algorithms we let 
$\,\U\,$ to depend on $\,\x_0\,$ and then this approach generally 
does not produce a symplectic map.

\section{Exponential integrators} \label{Sec:Integrators}

The use of approximations of the matrix exponential as a part of time propa\-gation
methods for differential equations has turned out to be very effective 
(see e.g.~\cite{Hochbruck_Lubich_Selhofer}).
We will consider application of three second order exponential integrators to the Hamiltonian system \eqref{eq:Ham_sys}.
In what follows, the matrix $\H$ will denote the Jacobian of the right hand side of \eqref{eq:Ham_sys} at $\x_0$, i.e.,
$\H = D \f(\x_0)$. The methods can be seen as exponential integrators applied to
semilinear equations which are local linearizations of \eqref{eq:Ham_sys}. 
In the literature methods of this type are also called \it exponential Rosenbrock methods\rm~\cite{Hochbruck_Ostermann}.

\subsection{Exponential Euler method} \label{subsec:EE}

As a first method we consider the exponential Euler method (EE)  
\footnote{We use the
shorthand notation $\,\x=\x_j\,,\ \x_+=\x_{j+1}\,$ etc.} 
\begin{equation}
  \label{eq:exp_E}
  \x_+=\x+h\,\phi(h\,\H)\,\f(\x)\ ,
\end{equation}
where $\,\phi(z)=\int_0^1 e^{tz}\,dt=z^{-1}\,(e^z-1)\,$ and $\,\H= D\f(\x)\,$.

Note that if $\,\f\,$ is linear, then $\,\x_+=\bphi^h(\x)\,$, i.e., 
the method gives exact values of the solution. 

Assume now that the system $\,\x'=\f(\x)\,$ is Hamiltonian in $\,\Rtn\,$ and
$\,\U\in\R^{2n\times2k}\,$ is symplectic. Then $\,\H=\D\f(\x)\,$
is a Hamiltonian matrix as well as $\,\F=\U^\dagger\H\U\,$ for $\,\U^\dagger=\J^{-1}\U^T\J\,$.

If we use the exponential Euler method \eqref{eq:exp_E} for the low dimensional system
\eqref{eq:Ham_sys_s} we produce
\begin{equation}
  \label{eq:exp_Euler}
  \x_+=\x+h\,\U\,\phi(h\,\F)\,\U^\dagger\f(\x)\ .
\end{equation}
For linear problems this will preserve energy exactly:
\begin{lem}\label{eEuler} Assume the system is of the form
$\,\f(\x)=\J^{-1}\,\nabla H(\x)=\H\x+\vec c\,$. 
Then the exponential Euler method \eqref{eq:exp_Euler} preserves energy,
i.e., $\,H(\x_+)=H(\x)\,$.  
\end{lem}

\begin{proof} The local problem now is (see \eqref{eq:Ham_sys_s})
\[\bxi'(t)=\F\,\bxi(t)+\U^\dagger\,(\H\x+\vec c)\ ,\qquad \bxi(0)=0\ .\]
Then $\,\bxi_+=\bxi(h)=h\,\phi(h\,\F)\,\U^\dagger\,(\H\x+\vec c)\,$, 
i.e., the exponential Euler approximation gives the exact solution 
for the problem in $\,\R^{2k}\,$. Hence
the energy is preserved in the small system and consequently also
for $\,\x_+=\x+\U\,\bxi(h)\,$.   
\end{proof}

\subsection{Explicit exponential midpoint rule} \label{subsec:EEMP}

We consider next the explicit exponential midpoint rule (see~\cite{Hochbruck_Lubich_Selhofer})
\begin{equation}
  \label{eq:eemp}
  \x_+=\x+e^{h\H}(\x_--\x)+2h\,\phi(h\H)\,\f(\x)\ ,
\end{equation}
where $\,\H=\D\f(\x)\,$.
For linear Hamiltonian problems $\,\f(\x)=\H\x+\vec c\,$ this gives
\begin{align*}
  \x_+&=\x+e^{h\H}(\x_--\x)+2\,(e^{h\H}-\I)\,\x+2h\,\phi(h\H)\,\vec c\\
&=e^{h\H}(\x_-+\x)-\x+2h\,\phi(h\H)\,\vec c\ ,
\end{align*}
i.e., $\,\half(\x_++\x)=e^{h\H}\half(\x+\x_-)+h\,\phi(h\H)\,\vec c=\tilde\x(h)\,$,
where $\,\tilde\x\,$ is the solution of $\,\tilde\x'(t)=\J^{-1}\,\nabla H(\tilde\x(t))\,$,
$\,\tilde\x(0)=\thalf(\x+\x_-)\,$. Hence the energy of
the \emph{averages} is preserved:
\begin{equation}
  \label{eq:eemp_energy}
  H(\thalf(\x_++\x))=H(\thalf(\x+\x_-))\ .
\end{equation}
\begin{remark}
  In~\cite{Eirola3} it was noticed that the explicit midpoint rule 
(for the homogeneous problem)
\[\x_+=\x_-+2h\,\H\x\]
preserves another quantity: $\,\omega(\H\x,\x_+)=\omega(\H\x_-,\x)\,$.
Equation \eqref{eq:eemp_energy} implies this, too.
\end{remark}

Again we approximate \eqref{eq:eemp} with
\begin{equation}
  \label{eq:eemp_K}
  \x_+=\x+\U\,e^{h\F}\,\U^\dagger(\x_--\x)+
2h\,\U\,\phi(h\F)\,\U^\dagger\f(\x)\ .
\end{equation}
For this we have
\begin{thm} Let $\,\f\,$ be linear, $\,\U\in\R^{2n\times2k}\,$
symplectic, and assume that $\,\x-\x_-\,$ is in the range of $\,\U\,$. 
Then \eqref{eq:eemp_energy} holds for the scheme \eqref{eq:eemp_K}.
\end{thm}
\begin{proof} Now $\,\U^\dagger\U=\I\,$. 
Write $\,\hat\x_+=\half(\x_++\x)\,,\ \hat\x=\half(\x+\x_-)\,$.
By the assumption there exists $\,\bzeta\in\R^{2k}\,$ such that 
$\,\half(\x-\x_-)=\U\bzeta\,$. Then $\,\x=\hat\x+\U\bzeta\,$. 
From \eqref{eq:eemp_K} and
$\,z\phi(z)=e^z-1\,$ we get
\begin{align*}
  \hat\x_+&=\x-\U\,e^{h\F}\,\bzeta+
h\,\U\,\phi(h\F)\,\U^\dagger\big(\H(\hat\x+\U\bzeta)+\vec c\big)\\
&=\hat\x+h\,\U\,\phi(h\F)\,\U^\dagger(\H\hat\x+\vec c)+
   \U\,[\I-e^{h\F}+h\,\phi(h\F)\F]\,\bzeta\\
&=\hat\x+h\,\U\,\phi(h\F)\,\U^\dagger(\H\hat\x+\vec c)\ .
\end{align*}
Thus the $\,\hat\x\,$-vectors propagate according to the exponential Euler method
\eqref{eq:exp_Euler} and we get the result by Lemma \ref{eEuler}.
\end{proof}

In case the column space of $\,\U\,$ contains the vector $\,\x_--\x\,$, we also have the following.

\begin{lem} \label{lem:eemp_symmetry}
Assume that $\,\U\,$ is a full rank matrix at $\,\x\,$ with a left inverse 
$\,\U^\dagger\,$, and that $\,R(\U)\,$ contains $\,\x_--\x\,$.
Then, the approximate explicit exponential midpoint rule is symmetric.
\end{lem}
\begin{proof}
Multiplying \eqref{eq:eemp_K} with 
$\,\U\,e^{-h\F}\,\U^\dagger\,$ gives
\[\U\,e^{-h\F}\,\U^\dagger(\x_+-\x)=\x_--\x+
2h\,\U\,e^{-h\F}\,\phi(h\F)\,\U^\dagger\f(\x)\ .\]
Since $\,e^{-z}\phi(z)=\phi(-z)\,$ we get
\[\x_-=\x+\U\,e^{-h\F}\,\U^\dagger(\x_+-\x)-
2h\,\U\,\phi(-h\F)\,\U^\dagger\f(\x)\ .\]
Thus the steps backward can be taken by replacing $\,h\,$ with $\,-h\,$. \\
\end{proof}

\subsection{Implicit exponential midpoint rule} \label{subsec:IEMP}

As a third exponential integrator, we consider implicit exponential midpoint rule (IEMP)
\begin{equation} \label{eq:imp}
\begin{aligned}
  0 &=e^{h\H}(\x - \hat\x)+ h\,\phi(h\H)\,\f(\hat\x)\ , \\
  \x_+ &= \hat\x  + e^{2 h\H}(\x - \hat\x) + 2 h\,\phi(2 h\H)\,\f(\hat\x)\
\end{aligned}
\end{equation}
(see~\cite{Celledoni_et_al}). This gives a symmetric method when the linear part $\, \H \,$ of $\, \f \,$ is fixed.
When $\H$ comes from a linearization of a nonlinear Hamiltonian system \eqref{eq:Ham_sys}, 
 the method is symmetric if $\, \H = D\f(\hat\x)\,$, where $\hat\x$ satisfies \eqref{eq:imp}.

For linear systems of the form $\, \f(\x) = \H \x + \vec c \,$, the second equation of \eqref{eq:imp} can be written
equivalently as $\, \x_+ = \hat\x  + 2 h\,\phi(2 h\H)\,\f(\x)\, $. Then, 
$\,\x_+$ propagates according to the exponential Euler method and
the energy is preserved in case $\,\U\,$ is symplectic (Lemma~\ref{eEuler}).

When we apply \eqref{eq:imp}, the total approximation is symmetric if 
$\, \H \,$ is evaluated at the midpoint $\hat \x$.

\begin{lem}
Assume that $\,\U\,$ is a full rank matrix with a left inverse $\, \U^\dagger \,$.
Suppose $\, \H \,$ is evaluated at $\hat\x$, where $\hat \x$ satisfies \eqref{eq:imp}.
Consider the approximation
\begin{equation} \label{eq:symm_step}
\, \x_+ = \x + \U \bxi_+,
\end{equation}
where $\,\bxi\,$ is obtained from applying \eqref{eq:imp} to the local system. Then, \eqref{eq:symm_step} gives a symmetric method.

\end{lem}
\begin{proof}
Applying \eqref{eq:imp} to the local system \eqref{eq:Ham_sys_s} gives
\begin{equation} \label{eq:imp_local}
\begin{aligned}
   0 &= - e^{h\H} \bxi + h\,\phi(h\F) \, \U^\dagger \f( \x + \U \bxi) \\
 \bxi_+ &= \bxi - e^{2h\F} \bxi + 2h \, \phi(2h\F) \, \U^\dagger \f( \x + \U \bxi). \\
\end{aligned}
\end{equation}
We show that \eqref{eq:imp_local} leads to a symmetric approximation of the full system.
Multiplying the upper equation of \eqref{eq:imp_local} by $\, e^{h\F}\,$, and using the relation
$\,e^z \phi(z) = 2\phi(2z) - \phi(z)\,$ gives
\begin{equation*}
- e^{h\H}\bxi + 2 h\,\phi(2 h\F) \, \U^\dagger \f( \x + \U \bxi) = h\,\phi( h\F) \, \U^\dagger \f( \x + \U \bxi)\ .
\end{equation*}
Combining this and both equations of \eqref{eq:imp_local} gives
\begin{equation*}
\bxi_+ = \bxi +  e^{h \F} \bxi. \\
\end{equation*}
Multiplying this from left by $\, \U \,$ and adding $\, \x \,$ gives
\begin{equation*}
\x_+ - \hat\x =  \U e^{h \F} \U^\dagger (\hat\x - \x) \ , \\
\end{equation*}
where $\, \hat \x = \x + \U  \bxi \,$ and $\, \x_+ = \x + \U  \bxi \,$. Replacing here $h$ with $-h$ and multiplying 
from the left by $ \, \U e^{h \F} \U^\dagger \,$ shows the symmetry. 
\end{proof}

Crucial for the symmetry of EIMP is that the Jacobian $\,\H \,$ and the basis $\, \U \,$ are the same when considering
stepping from $\, \x \,$ to $\, \x_+ \,$ and vice versa.
This is the case if $\, \H \,$ is evaluated at $\, \hat\x \,$, and $\, \U \,$ is generated using the Krylov subspace methods
described in~\ref{Sec:Krylov}.
%

Our numerical strategy is to perform one time step $h$ using the exponential Euler method 
from $\, \x \,$ to $\, \widetilde \x \, $ in order to approximate the midpoint $\, \hat \x \,$. Then after evaluating the Jacobian
$\, \H \,$ and forming the basis $\, \U \,$ at $\, \widetilde \x \, $ we solve the implicit equation
using fixed point iteration and perform the step of size $\, 2h \,$ to obtain $\, \bxi_+ \,$ and
$\, \x_+ = \x + \U \bxi_+ \,$.

\section{Forming the local basis using Krylov subspace methods} \label{Sec:Krylov}

We discuss next the approximation of matrix valued $\phi$ functions using Krylov subspace methods and
show how they are naturally connected to the local approximation discussed in Section~\ref{Sec:Integrators}.

When matrix $\,\A\,$ is large but the operation $\,\v\to\A\v\,$ inexpensive, 
it is reasonable to use Krylov subspace methods. These work in Krylov subspaces
$$
K_k(\A,\v)=\Span\{\v,\A\v,\A^2\v,\dots ,\A^{k-1}\v\}.
$$
Then we have $\,\A\,K_k(\A,\v)\subset K_{k+1}(\A,\v)\,$. The Arnoldi iteration 
uses the Gram-Schmidt process and produces an orthonormal basis 
$\,\q^1,\dots,\q^k\,$ for $\,K_k(\A,\v)\,$. 
Denote $\,\Q_k=[\q^1 \dots \q^k]\,$ and $\,\F_k=\Q_k^T\A\Q_k\,$, which is
a Hessenberg matrix. 

If the iteration stops, i.e. $\,\A^k\v\in K_k(\A,\v)\,$, then
\begin{itemize}
\item[a)] $\,\A K_k(\A,\v)\subset K_k(\A,\v)\,$ and $\, K_j(\A,\v)= K_k(\A,\v)\,$
for all $\,j>k\,$, 
\item[b)]  $\,\A\Q_k=\Q_k\F_k\,$ and for the spectra we have 
$\,\Lambda(\F_k)\subset\Lambda(\A)\,$, 
\item[c)] If $\,\varphi(z)=\sum_ja_j\,z^j\,$ has convergence radius
larger than the spectral radius of $\,\A\,$ and $\,\w\in R(\Q_k)\,$, then
\[\varphi(\A)\w=\Q_k\,\varphi(\F_k)\,\Q_k^T\w\ .\]
\end{itemize}
The effectivity of Krylov subspace methods is based on the fact that 
if the component of $\,\A^k\v\,$ orthogonal to $\, K_k(\A,\v)\,$ is small, 
then things are approximately as 
above and this can happen already for a reasonable size $\,k\,$.
Thus it is reasonable to consider the approximation
\begin{equation} \label{eq:Krylov_approx}
\varphi(\A) \v = \Q_k\,\varphi(\F_k)\,\Q_k^T\w\,
\end{equation}
which was used already in~\cite{Druskin_Knizhnerman} and~\cite{Gallopoulos_Saad}. We refer to~\cite{Hochbruck_Lubich}
for a detailed error analysis.

We show next how the Krylov approximation \eqref{eq:Krylov_approx} is naturally connected to 
the strategy of applying exponential integrators to the local system \eqref{eq:Ham_sys_s}.

\subsection{Equivalence of the Krylov and the local system approximations}

Consider the local system \eqref{eq:Ham_sys_s} corresponding to the basis
$\U = \Q_k$, where $\Q_k$ gives an orthonormal basis for $K_k(\H,\f(\x_0))$. 
Recall from Section~\ref{Sec:Integrators} the strategy of solving the local system~\eqref{eq:Ham_sys_s}, i.e.,
$$
\bxi'(t)=\U^\dagger \f(\x_0+\U\bxi(t))\,
$$
numerically from $\,\bxi_0=0\,$ up to $\,\bxi_1\approx\bxi(t_1)\,$ and setting $\,\x_1=\x_0+\U\bxi_1\,$. 
As shown in Subsection~\ref{subsec:EE}, applying the exponential Euler method to the local system gives
the approximation
$$
\x_+ = \x + h \U \phi(h \F) \U^\dagger \f(\x).
$$
We immidiately see from \eqref{eq:Krylov_approx} that this is the Krylov subspace approximation of the
exponential Euler step \eqref{eq:exp_E}.

As shown in subsection~\ref{subsec:EEMP}, applying the exponential explicit midpoint rule gives
$$
  \x_+ = \x + \U\,e^{h\F}\,\U^\dagger(\x_- - \x) + 2h\,\U\,\phi(h\F)\,\U^\dagger\f(\x).
$$
As for the explicit Euler method, the resulting formula can be seen as a Krylov subspace approximation \eqref{eq:Krylov_approx}
of the EEMP step \eqref{eq:eemp}.
In order to satisfy the assumptions of Lemma~\ref{lem:eemp_symmetry} 
the vector $\x_- - \x$ has to be in the range of $\U$. This is addressed in Subsection~\ref{subsec:add_vector}.

Similarly, if we perform a Krylov approximation of the IEMP step \eqref{eq:imp}, and denote $\bxi = \hat \x - \x$ and $\bxi_+ = \x_+ - \x$,
we get the small dimensional system \eqref{eq:imp_local}. \\

The present concern, however, is that if $\,\A\,$ above is a Hamiltonian matrix, 
this does not necessarily bring  any special structure to $\,\Q_k\,$ or $\F_k$.


\subsection{Symplectic Krylov processes}

In order to obtain good local approximations for a Hamiltonian system with
linear part $\,\H\,$ we would like to have
\begin{itemize}
\item[a)] A symplectic subspace $\,\W\,$ with a corresponding basis.
\item[b)] $\,K_k(\H,\f)\subset\W\,$ in order to have polynomials of $\,\H\,$
applied to $\,\f\,$  represented in $\,\W\,$. We expect this to be worth
pursuing for approximations of $\,\varphi(\H)\f\,$. 
 \end{itemize}

Consider first the Krylov subspace corresponding to $\,\H\,$ and $\,\v\,$:
$$ K_k(\H,\v)=\Span\{\v,\H\v,\H^2\v,\dots ,\H^{k-1}\v\}\ ,$$
and set $\,W_k=K_k(\H,\v)+\J\,K_k(\H,\v)\,$.
\begin{itemize}\itemsep3pt
\item Now $\,\H\,W_k(\H,\v)\not\subset W_{k+1}(\H,\v)\,$, 
generally.
\item If $\,p\,$ is a degree $\,k-1\,$ polynomial, then 
$\,p(\H)\v\in K_k(\H,\v)\subset W_k\,$. 
\end{itemize}
The construction of a symplectic basis for $\, W_k\,$ is
slightly more complicated than the standard Arnoldi process:
\medskip

We just reorthogonalize with respect to
$\,\inprod{\,\cdot\,,\,\cdot\,}\,$ and $\,\omega(\,\cdot\,,\,\cdot\,)\,$ 
-- the $\,\inprod{\,\cdot\,,\,\cdot\,}\,$\break --orthogonal vectors provided by the
standard Arnoldi. The result is a symplectic and orthonormal matrix.
\begin{enumerate}
\item \quad $\,\q=\v/\norm{\v}\,,$ $\,\Q=[\q]\,$ $\,\V=\Q\,$ 
\item \quad for $\,j=2,\dots,k\,$ do \qquad $\,\r=\H\,\q\,$ \\ 
\null\quad\quad $\,\r\leftarrow\r-\Q\Q^T\r\,$,\\ 
\null\quad\quad if $\,\r\ne0\,$, set $\,\q=\r/\norm{\r}\,$, 
\qquad$\Q\leftarrow[\Q\,,\,\q]\,$,\\
\null\quad\quad\quad $\,\r=\q-\V\V^T\q-\J\V\V^T\J^T\q\,$,\\
\null \quad\quad\quad $\,\V\leftarrow[\V\,,\,\r/\norm \r]\,$.\\  
\null\quad\quad else stop.
\item\quad Set $\,\U=[\V\ \J\V]\,$, $\,\F=\U^T\H\U\,$.
\end{enumerate}
Here the columns of $\,\Q\,$ form an orthonormal basis for $\, K_k(\H,\v)\,$
and those of $\,\U\,$ a symplectic basis for $\,W_k\,$. 
\medskip

\begin{remark}
There is a way to construct matrix $\,\F\,$ more economically from the 
computations of step 2. but anyway the reorthogonalization stays the costly part
of this approach.
\end{remark}

\noindent\textbf{Isotropic Arnoldi}
\smallskip

Mehrmann and Watkins \cite{Mehrmann_Watkins}
suggest the \emph{isotropic Arnoldi process}, which is a direct
$\,\inprod{\,\cdot\,,\,\cdot\,}\,$ and $\,\omega(\,\cdot\,,\,\cdot\,)\,$ 
--orthogonalization in the Arnoldi process:
\begin{enumerate}
\item \quad $\,\q=\v/\norm{\v}\,,$ $\,\Q=[\q]\,$ 
\item \quad for $\,j=2,\dots,k\,$ do \qquad $\,\r=\H\,\q\,$ \\ 
\null\quad\quad $\,\r\leftarrow\r-\Q\Q^T\r-\J\Q\Q^T\J^T\r\,$,\\ 
\null\quad\quad if $\,\r\ne0\,$, set $\,\q=\r/\norm{\r}\,$, 
\qquad$\Q\leftarrow[\Q\,,\,\q]\,$,\\
\null\quad\quad else stop.
\item\quad Set $\,\U=[\Q\ \J\Q]\,$, $\,\F=\U^T\H\U\,$.
\end{enumerate}
Here we obtain a symplectic matrix with orthonormal columns. 
However its range does not necessarily contain the Krylov subspace $\,K_k(\H,\v)\,$. 
Thus, generally, this iteration does not have the property that 
$\,p(\H)\v\,$ is in the span of $\,\u^1,\dots,\u^k\,$ for every 
polynomial $\,p\,$ of degree $\,k-1\,$. 
Since our present aim is to approximate operator functions
that have converging power series, this iteration can be expected to be
less effective for our purposes\footnote{Mehrmann and Watkins use the iteration
for a quite different purpose: eigenpairs of skew-Hamiltonian/Hamiltonian 
pencils.}. We will see this in the numerical tests.

Also there is a possibility of a breakdown: after orthogonalization we may 
get $\,\r=0\,$ without obtaining any useful information about $\,K_k(\H,\v)\,$.
\bigskip

\noindent\textbf{Hamiltonian Lanczos}\\
Benner, Fa\ss bender, and Watkins have several versions of Hamiltonian
Lanczos processes (see \cite{Benner_Fassbender}, \cite{Watkins}). 
The following is a promising one from Watkins 
(Algorithm 4 of \cite{Watkins}):
\begin{enumerate}
\item \quad $\,\u^0=0,\ \u^1=\v/\norm{\v}\,,\ \beta_0=0\,,$
\item \quad for $\,j=0,1,\dots,k\,$ do \\
\null\quad\quad if $\,j>0\,$  , then \\ 
\null\quad\quad\quad$\,\x=\H\,\u^j\,,\ \alpha_j=\inprod{\J\v^j,\x_j}$ \\
\null\quad\quad\quad$\,\u^{j+1}=\x-\alpha_j\u^j-\beta_{j-1}\u_{j-1}\,$ \\
\null\quad\quad if $\,j<k\,$  , then \\
\null\quad\quad\quad 
$\,\v^{j+1}=\H\u^{j+1}\,,\ \tau=\inprod{\J \v^{j+1},\u^{j+1}}\,$  \\
\null\quad\quad\quad if $\,\tau\approx 0\,$ then  stop (breakdown).\\
\null\quad\quad\quad 
$\,\sigma=\sqrt{\abs\tau}\,,\ \delta_{j+1}=\sgn\tau\,$ \\
\null\quad\quad\quad $\,\u^{j+1}\leftarrow\u^{j+1}/\sigma\,$, 
$\,\v^{j+1}\leftarrow\delta_{j+1}\v^{j+1}/\sigma\,$ \\
\null\quad\quad\quad if $\,j>0\,$  , then $\,\beta_j=\sigma\,$ 
\item\quad Form the matrices\\[3pt] 
\null\quad\quad$\,\U=[\u^1,\dots,\u^k,\v^1,\dots,\v^k]\,$,
\null\quad\quad
$\,\F=\bmat{0&\mat T\\\mat D&0}\,,$\\[3pt]
where
\null\quad$\ \mat T=\smat{\alpha_1&\beta_1&&&\\\beta_1&\alpha_2&\beta_2&&\\
&\beta_2&\sddots&\sddots&\\&&\sddots&\sddots&\beta_{k-1}\\
&&&\beta_{k-1}&\alpha_k}\,$ and 
$\,\mat D=\smat{\delta_1&&&&\\&\delta_2&&&\\
&&\sddots&&\\&&&\sddots&\\&&&&\delta_k}\,.$ \\[6pt]
Then $\,\U\;:\;\R^{2k}\to\Rtn\,$ is symplectic, its range contains
the vectors $\,\H^j\,\v\,,\ j=0,\dots,2k-1\,$,
and  $\,\F=\U^\dagger\H\U\,$.
\end{enumerate}
Due to short recursion this is an economic iteration. 
But it has similar problems as the usual biorthogonal Lanczos,
e.g., near breakdowns and loss of orthogonality. These can be partly
circumvented. For small $\,k\,$ this may be a good choice.

By the very construction of these symplectic maps we get the following:

\begin{proposition}
  Combining any of the symplectic Krylov processes with a method that
preserves energy for the small dimensional system \eqref{eq:Ham_sys_s}
will preserve the energy of the original system, too.
\end{proposition}
\medskip

In the numerical experiments we will use four algorithms to produce
Krylov subspaces: the standard Arnoldi iteration in $\,\Rtn\,$ and
the three symplectic ones: symplectic Arnoldi, isotropic Arnoldi, and
the Hamiltonian Lanczos process. With respect to their costs to
produce a space of fixed dimension they can be ordered as
\smallskip

Hamiltonian Lanczos $\,<\,$ Arnoldi $\,<\,$ Isotropic Arnoldi $\,<\,$
Symplectic Arnoldi
\bigskip

The main weaknesses of each of these are
\begin{itemize}\itemsep1pt
\item Arnoldi in $\,\Rtn\,$: the approximation is not Hamiltonian.
\item Hamiltonian Lanczos: breakdown, early loss of symplecticity.
\item Isotropic Arnoldi: does not include a Krylov subspace.
\item Symplectic Arnoldi: expensive.
\end{itemize}

\subsubsection{Adding a vector to the basis} \label{subsec:add_vector}

When using the EEMP method \eqref{eq:eemp}, the vector $\u_{-1} - \u_0$ needs to be added to the basis $\U_k$
at each time step.
For orthogonal and/or isotropic basis this is straightforward. For the symplectic basis,
a symplectic version of the Gram--Schmidt algorithm adds $\x$ and $\J \x$ to the basis
$\U = [\V \,\, \W]$. This algorithm is shown in the following pseudocode.
Here we denote $\omega(\cdot, \cdot) = \langle \J \cdot, \cdot \rangle$.
Here the symplectic orthogonalization can also be performed in modified Gram-Schmidt manner, one vector at a time.
Notice also that in the second step the vector $\hat{\x}$ can be scaled with any constant.

\begin{algorithm}
\caption{Symplectic reorthogonalization of $\x \in \mathbb{R}^{2n}$ onto $[ \V \,\, \W] \in \mathbb{R}^{2n \times 2k}$. }
\begin{algorithmic}
\STATE{$ \hat{\x} = \x - \sum\limits_{k=1}^{n} \omega(\w_k,\x) \v_k +   \sum\limits_{k=1}^{n} \omega(\v_k,\x) \w_k$}
\STATE{$ \u_{k+1} \leftarrow \hat{\x}$}
\STATE{$ \x \leftarrow \J \hat{\x}$}
\STATE{$ \tilde{\x} = \x - \sum\limits_{k=1}^{n} \omega(\w_k,\x) \v_k +   \sum\limits_{k=1}^{n} \omega(\v_k,\x) \w_k$}
\STATE{$ \w_{k+1} \leftarrow - \frac{\tilde{\x}}{\omega(\v_{k+1}, \tilde{\x})}$}
\end{algorithmic}
\end{algorithm}

\section{Numerical tests}

We compare numerically the three exponential time integrators of Section~\ref{Sec:Integrators}
and the four Arnoldi like processes of Section~\ref{Sec:Krylov} to produce the local basis $\U_k$.
We apply the methods to large sparse Hamiltonian systems which are obtained from
finite difference discretizations of one dimensional nonlinear wave equations.
For ease of presentation we first illustrate by an example our approach of deriving large sparse Hamiltonian systems 
from Hamiltonian PDEs.
For further examples we refer to~\cite{Celledoni_et_al2}.

\subsection{Spatial discretization of Hamiltonian PDEs} \label{subsec:derivation}
 As an example consider the nonlinear Klein-Gordon equation in one dimension,
\begin{equation} \label{eq:wave_equation}
  u_{tt} = u_{xx} - f(u),
\end{equation}
where $u(x,t)$ is scalar valued periodic function $(u(0,t) = u(L,t)$ for some $L>0$) and $f$ is a smooth function. 
Setting $v = u_t$ and $\u = (u , v)^T$, the equation \eqref{eq:wave_equation} can be viewed as a Hamiltonian system
\begin{equation*}
   \mat J \u_t = \frac{\delta H}{\delta \u}, 
\end{equation*}
where
\begin{equation*}
    \quad \mat J = \begin{bmatrix} 0  &  1  \\ -1 & 0 \end{bmatrix} ,  \quad \frac{\delta H}{\delta \u} = \left( \frac{\partial H}{\partial u} ,\frac{\partial H}{\partial v}\right)^T, 
\end{equation*}
and $\frac{\partial H}{\partial u}$ and $\frac{\partial H}{\partial v}$ denote the functional derivatives of the Hamiltonian
\begin{equation} \label{eq:Hamiltonian}
    H(\u) = \int\limits_0^L \left[ \frac{1}{2} v^2 - \frac{1}{2} u_x^2 + F(u) \right] dx, \quad F'(u) = f(u).
\end{equation}
To obtain a Hamiltonian system approximating the equation \eqref{eq:wave_equation}, we 
perform a discretization with respect to the spatial variable $x$
on an equidistant grid with an interval $\Delta x  = L/n$, $n \in \mathbb{N}^+$, and denote by $q_i(t)$ and $p_i(t)$ $i=1,\ldots,n$, 
the approximations to $u(i\Delta x,t)$ and $v(i\Delta x,t)$. For the second derivative $u_{xx}$ we use the central difference approximation
\begin{equation*}
      u_{xx} (i \Delta x,t) \approx \frac{q_{i-1}(t) - 2q_i(t) + q_{i+1}(t)}{(\Delta x)^2}.
\end{equation*}
Expressing the approximations as vectors
\begin{equation*}
    \q(t) = \begin{bmatrix}
		  q_1(t) \\ \vdots \\ q_n(t)
              \end{bmatrix} \quad \textrm{and} \quad  \p(t) = \begin{bmatrix}	p_1(t) \\ \vdots \\ p_n(t) \end{bmatrix},
\end{equation*}
we get the approximation of the PDE in matrix form as
\begin{equation*}
	\p'(t) = {\boldsymbol \Delta}_n \q(t) - \f(\q(t)),
\end{equation*}
where $\{\f(\q)\}_i = f(q_i)$ and $\boldsymbol \Delta_n \in \mathbb{R}^{n \times n}$ is the discretized Laplacian with periodic boundary conditions,
\begin{equation} \label{eq:laplacian_periodic}
  \boldsymbol \Delta_n = n^2
 \begin{bmatrix}
    -2 & 1 & & &1 \\
    1 & -2 & 1 & & \\
      & \ddots & \ddots & \ddots & \\
      & & 1& -2& 1 \\
     1 & & & 1 & -2 \end{bmatrix}.
\end{equation}
Defining  the Hamiltonian function
\begin{equation} \label{eq:discrete_Hamiltonian}
	H(\q,\p) = \frac{1}{2} \p^T  \p  - \frac{1}{2} \q^T \boldsymbol \Delta_h \q + \sum\limits_{i=1}^n F(q_i),
\end{equation}
we see that
\begin{equation*}
	\p'(t) = -\nabla_{\q} H(\q(t),\p(t)),
\end{equation*}
and by setting $\x(t) = \smat{ \q(t) \\ \p(t) }$, we have the Hamiltonian system 
\begin{equation} \label{eq:Hamiltonian_system}
    \x'(t) = \mat J^{-1} \nabla H(\x(t)), \quad \x(0) = \x_0,
\end{equation}
where $\J=\smat{0&I\\-I&0}$, $\x_0 = \smat{\u_0 \\ \v_0 }$, 
and $\u_0$ and $\v_0$ come from the discretizations
of the initial values of \eqref{eq:wave_equation}. Notice that \eqref{eq:discrete_Hamiltonian} is
the discrete counterpart of \eqref{eq:Hamiltonian}.


\subsection{Linear wave equation} \label{subsec:linear_wave}


As a first numerical example we consider the linear wave equation with periodic boundary conditions,
\begin{equation*}
  \begin{aligned}
      \partial_t u(x,t) &= \partial_{xx} u(x,t) + f(x) \\
      u(x,0) &= u_0(x), \quad u_t(x,0) = v_0(x)  \\
      u(0,t) &= u(L,t) = 0, 
  \end{aligned}
\end{equation*}
where $x\in [0,L]$, $t\in [0,T]$, and
$$
f(x) = \frac{1}{8} \big(x (x-L)\big)^2, \quad u_0(x) = \frac{1}{1 + \sin^2(\pi x)}-1, \quad v_0(x) = 0.
$$

Performing spatial discretization on an equidistant grid of size $n$ using central
differences leads to a Hamiltonian system of the form \eqref{eq:Hamiltonian_system} in $\mathbb{R}^{2n \times 2n}$
with the Hamiltonian
\begin{equation*}
H\big(\x(t) \big) = \frac{1}{2} \q(t)^T \Delta_n \; \q(t) - \frac{1}{2} \norm{\p(t)}^2 + \bc^{\;T} \q(t)
\end{equation*}
and initial data $ \x(0) = \smat{ \u_0 \\ \v_0}$.
Here $\bc_i = f(x_i)$, $(\u_0)_i = u_0(x_i)$ and $(\v_0)_i = v_0(x_i)$, where $x_i = i \Delta x$, and
\begin{equation*}
\Delta_n = n^2 \begin{bmatrix}
    -2 & 1 & & \\
     1 & \ddots & \ddots & \\
      & \ddots & \ddots & 1 \\
      & & 1 & -2 \end{bmatrix}.
\end{equation*}
We set $L=2$, $n = 400$, $\Delta x = L/n$, and we integrate up to $T=50$ with time step size $h = T/n_t$, where $n_t=2000$.

Using this linear example we illustrate the differences between the iterative processes of Section~\ref{Sec:Krylov}
to produce the basis $\U_k \in \mathbb{R}^{2n \times 2k}$. We apply the
exponential Euler method \eqref{eq:exp_Euler} to the small dimensional system~\eqref{eq:Ham_sys_s} obtained from
the projection using $\U_k$. Note that for linear systems all the three integrators of Section~\ref{Sec:Integrators} propagate as the exponential
Euler method.

As illustrated in Figures~\ref{fig:linear_energies1}, the approximation obtained using the Arnoldi iteration results generally in a
linear growth of the energy error, whereas the symplectic basis gives a bounded energy error.
Figure~\ref{fig:linear_energies2} shows that, as opposed to the Hamiltonian Lanczos approximation, 
the energy error of the Arnoldi approximation is dependent on the accuracy of the approximation.
Notice that in both cases
\begin{equation} \label{eq:K_property} 
K_\ell(\H,\f_0) \subset \mathrm{Range}(\U_k), \quad \textrm{where} \quad \ell = \dim(\U_k),
\end{equation}
and $\f_0$ is the right hand side of \eqref{eq:Hamiltonian_system} evaluated at $x(0)$.
Property \eqref{eq:K_property} means that these processes give a polynomial approximation of degree $\ell$ 
for the exponential Euler step which gives the exact solution at $t=h$. This effect is also seen in Figure~\ref{fig:linear_solutions1},
which depicts the solution errors for the Arnoldi iteration and  the Hamiltonian Lanczos process.
When $\dim(\U_k) = 16$, the methods give errors not far from each other, however for smaller basis size
the symplectic alternative gives more accurate results.

When increasing the basis size also the isotropic Arnoldi and the symplectic Arnoldi
start to perform better (see Figure~\ref{fig:linear_solutions2}).
Need for a larger dimension is expected for the symplectic Arnoldi since instead of \eqref{eq:K_property}, only
$K_{\ell/2}(\H, \f_0) \subset \mathrm{Range}(\U_k)$, where $\ell= \dim(\U_k)$.
Isotropic Arnoldi performs worse as expected due to its poor polynomial approximation properties.
However, both processes give bounded energy errors as in both cases $\U_k$ is symplectic (Figure~\ref{fig:linear_energies3}).

\newpage

\begin{figure}[h!]
\begin{center}
\includegraphics[scale=0.4]{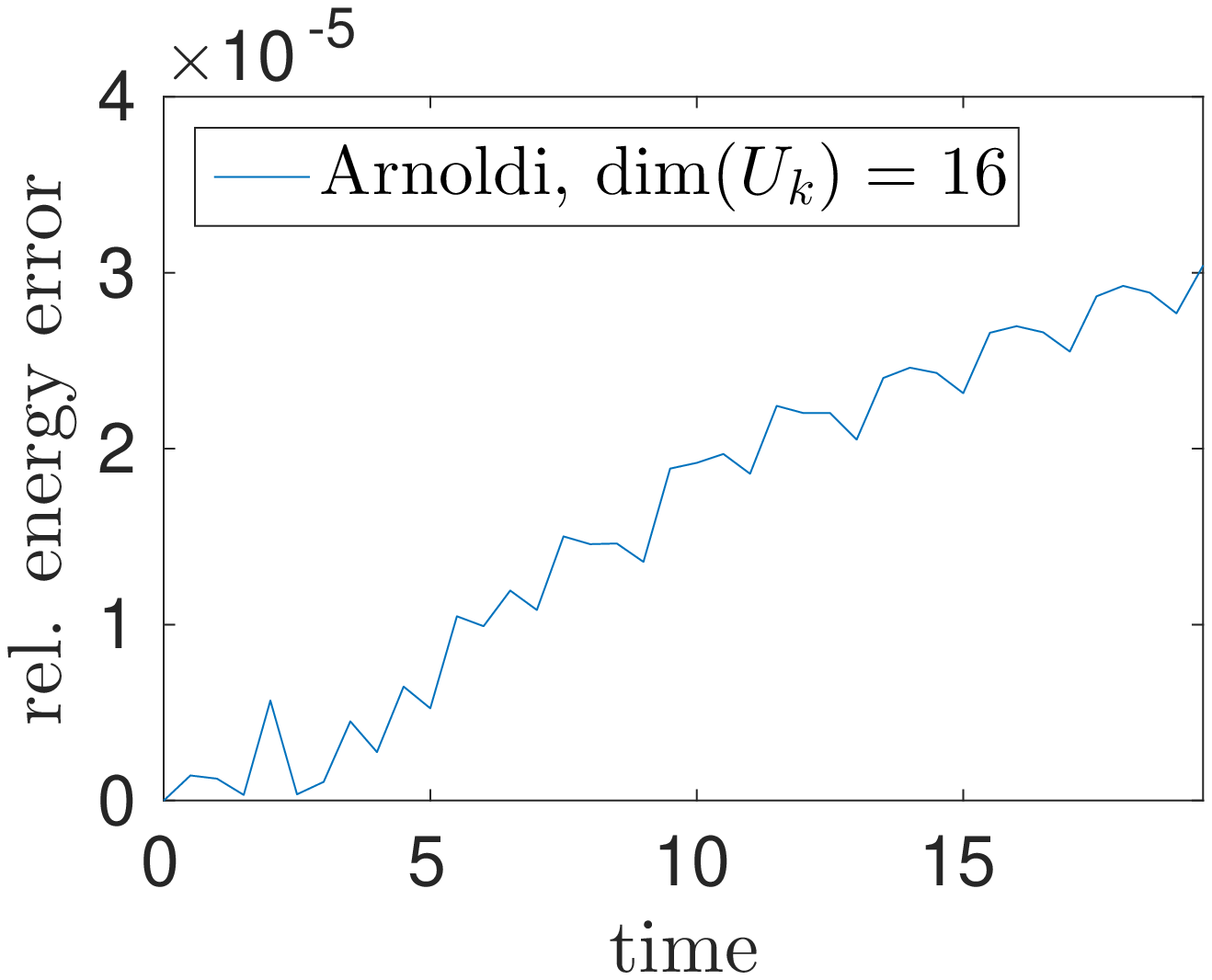}\hspace{-0mm}
\includegraphics[scale=0.4]{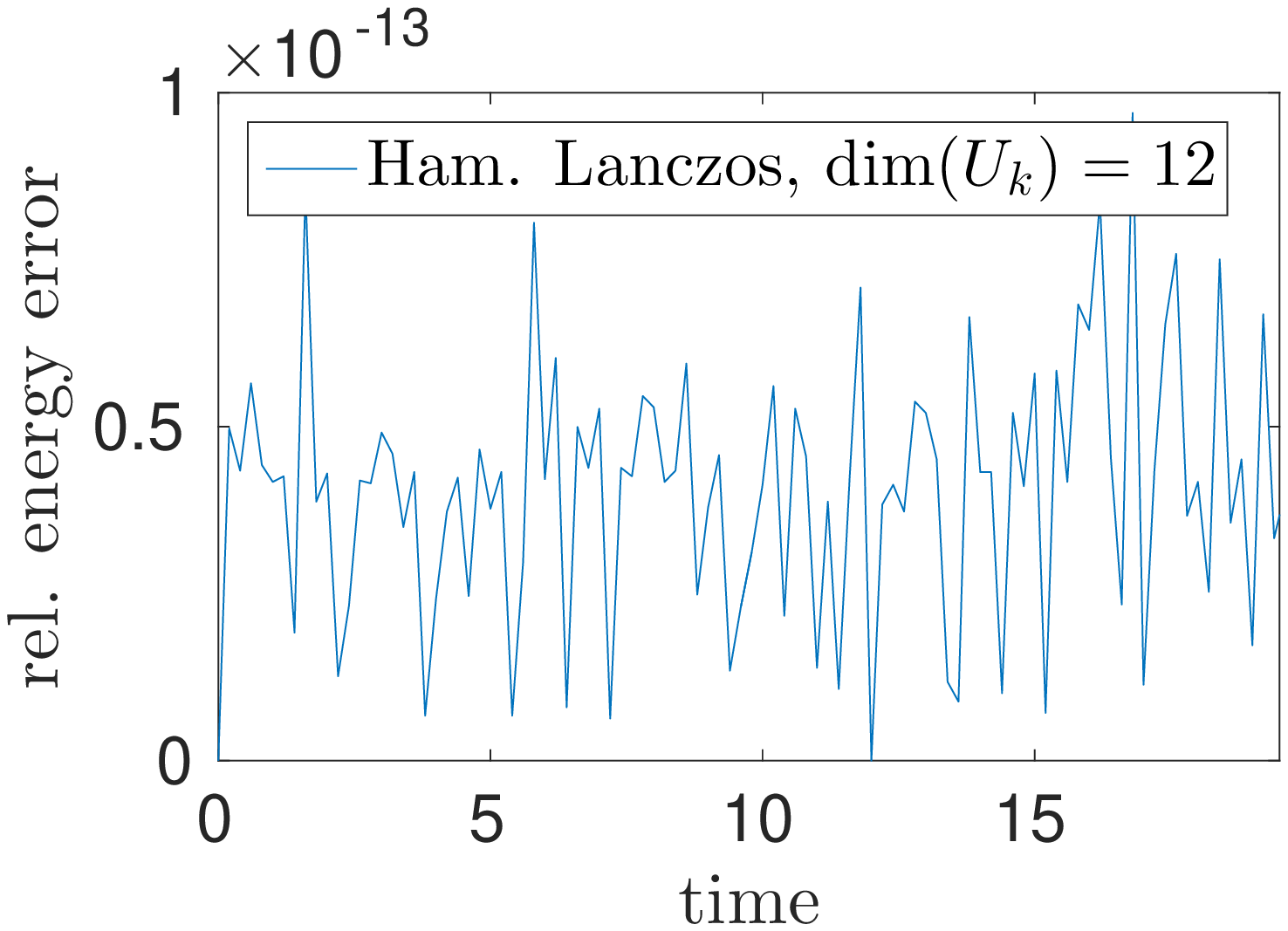}
\end{center}
\caption{Linear wave equation and relative energy errors for the exponential Euler method, 
when $\U_k$ is produced by the Arnoldi iteration
($\U_k \in \mathbb{R}^{2n \times 16}$) and the Hamiltonian Lanczos process ($\U_k \in \mathbb{R}^{2n \times 12}$).}
\label{fig:linear_energies1}
\end{figure}

\begin{figure}[h!]
\begin{center}
\includegraphics[scale=0.5]{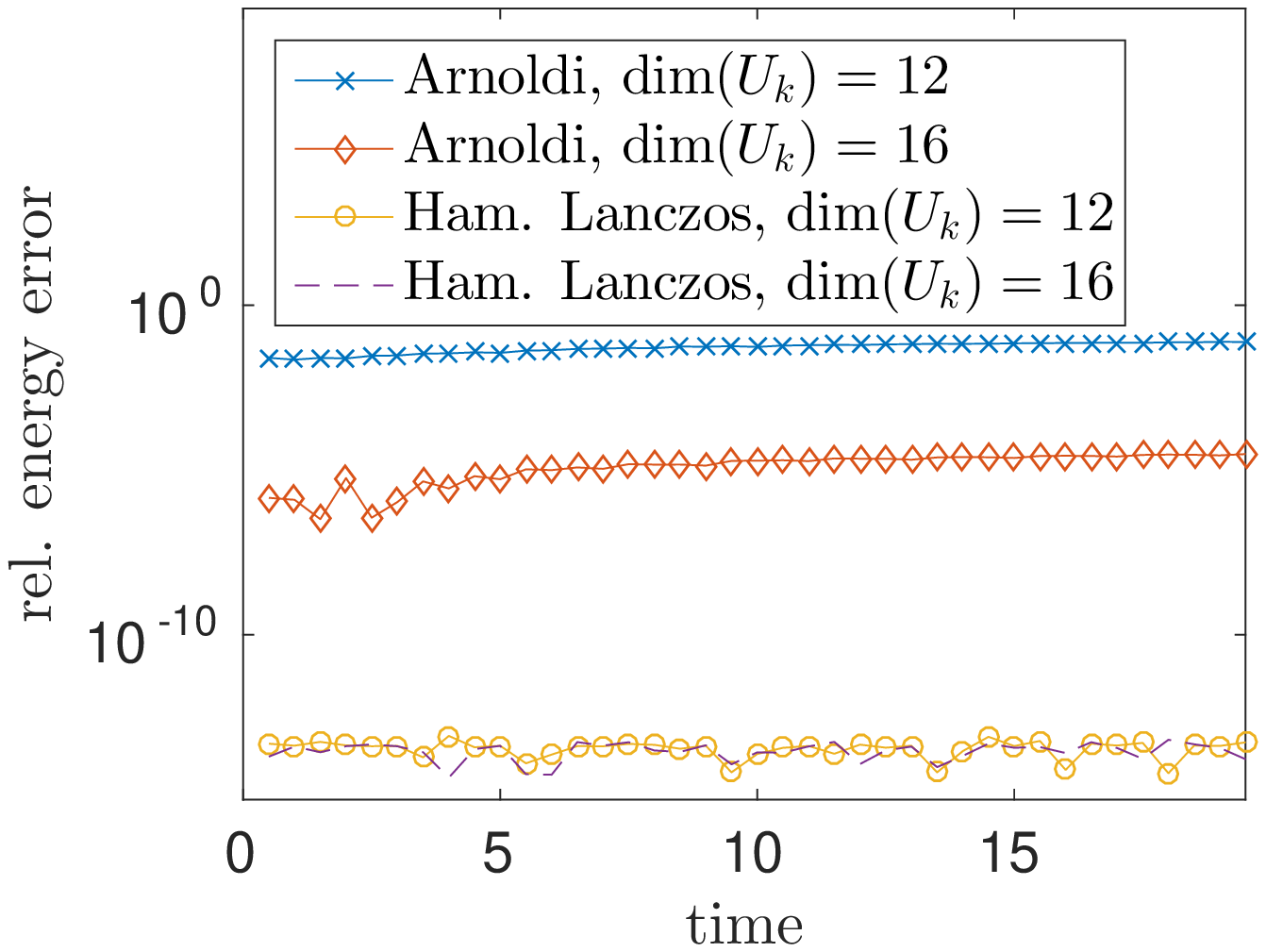}\hspace{-0mm}
\end{center}
\caption{Linear wave equation and energy errors for EE, 
when $\U_k$ is produced by the Arnoldi iteration and the Hamiltonian Lanczos process. }
\label{fig:linear_energies2}
\end{figure}

\begin{figure}[h!]
\begin{center}
\includegraphics[scale=0.44]{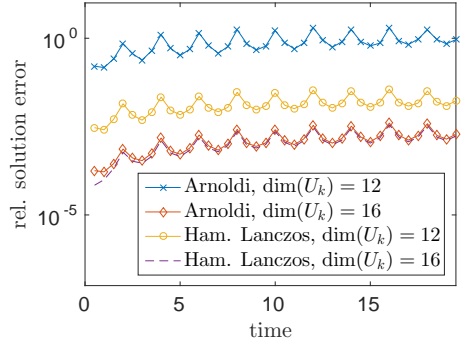}\hspace{-0mm}
\end{center}
\caption{Linear wave equation and solution errors when $\U_k$ is produced by 
the Arnoldi iteration and the Hamiltonian Lanczos process. }
\label{fig:linear_solutions1}
\end{figure}

\begin{figure}[h!]
\begin{center}
\includegraphics[scale=0.5]{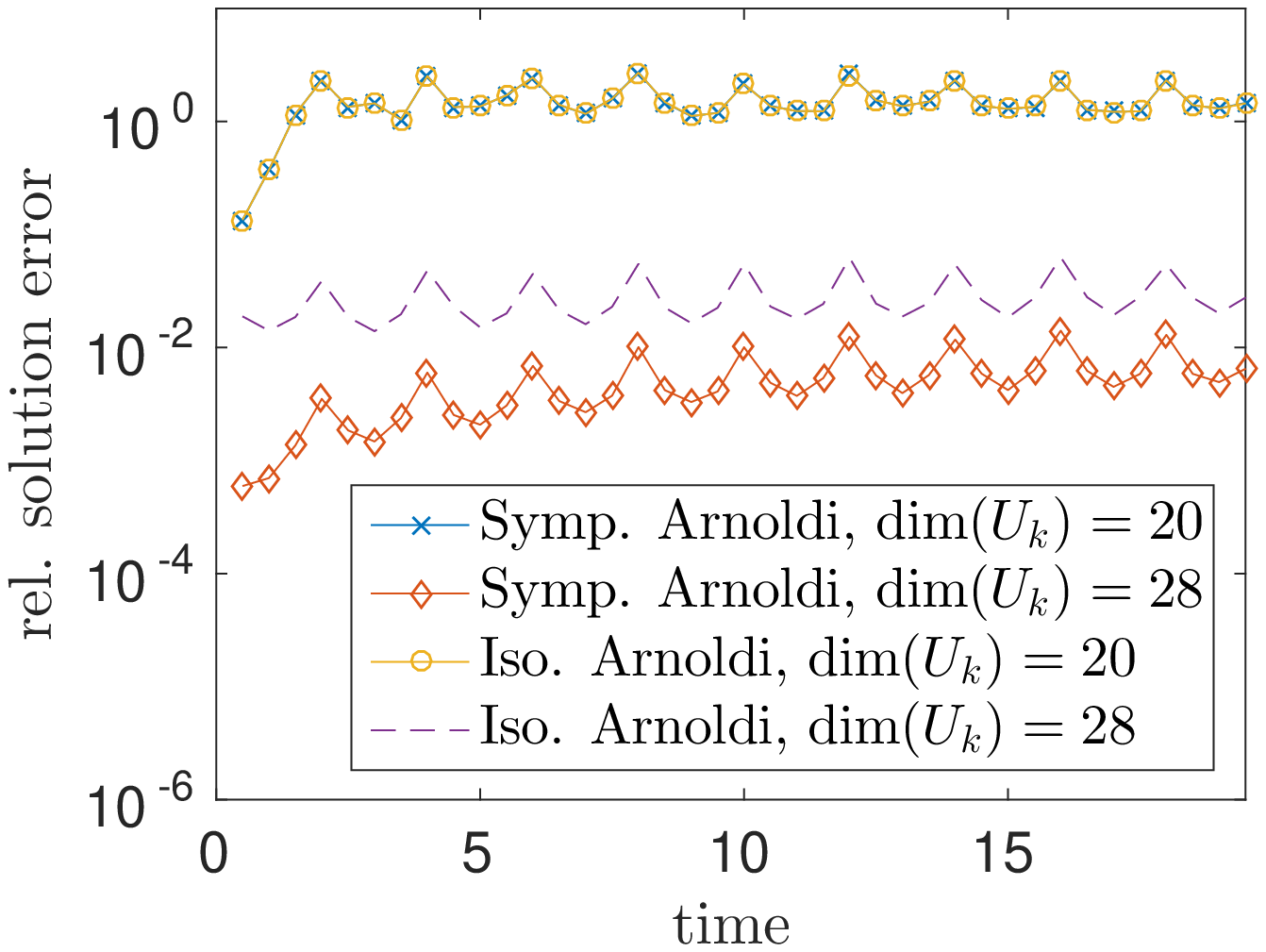}\hspace{-0mm}
\end{center}
\caption{Linear wave equation and solution errors when $U_k$ is produced by 
the symplectic Arnoldi iteration and the isotropic Arnoldi process. }
\label{fig:linear_solutions2}
\end{figure}

\begin{figure}[h!]
\begin{center}
\includegraphics[scale=0.5]{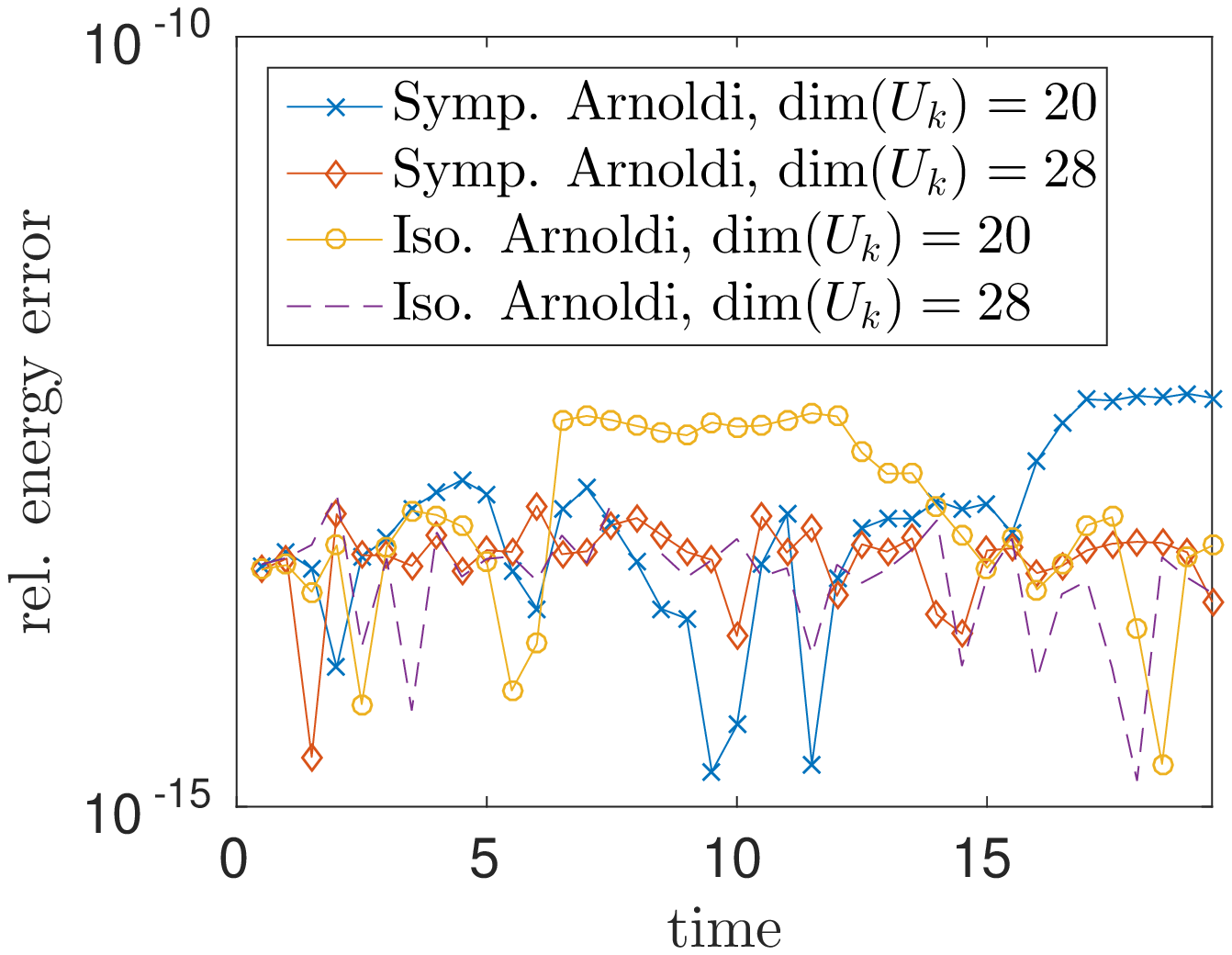}\hspace{-0mm}
\end{center}
\caption{Linear wave equation and energy errors when $U_k$ is produced by 
the symplectic Arnoldi iteration and the isotropic Arnoldi process. } 
\label{fig:linear_energies3}
\end{figure}

\newpage


\subsection{Nonlinear Schr\"odinger equation} \label{subsec:nls}


Consider next a one dimensional nonlinear Schrödinger equation (NLS) on $[-4\pi,4\pi]$ with periodic boundary conditions,
\begin{equation} \label{eq:nls_system}
  \begin{cases}
      i \partial_t \psi(x,t) &= - \frac{1}{2} \, \partial_{xx} \psi(x,t) + \abs{\psi(x,t)}^2\psi(x,t) \nonumber - V_0 \sin^2 (x) \psi(x,t) \\
      \psi(x,0) &= \psi_0(x), \,\,\,\,\,\, \textrm{ for all } x\in [-4\pi, 4\pi] \tag{6.1} \\
      \psi(-4\pi,t) &= \psi(4\pi,t), \,\,\,\,\, 
  \end{cases}
\end{equation}
(see~\cite{Bronski_et_al}). The initial value is given by
\begin{equation} \label{eq:init_nls}
\psi_0(x) = \sqrt{V_0 \, \sin^2(x) + B} \,\, \ee^{\ii \theta(x)},
\end{equation} 
where the phase function $\theta(x)$ satisfies
$$
\tan(\theta(x)) = \pm \sqrt{1 + V_0/B} \, \tan (x)
$$
(see Figure~\ref{fig:theta}).
We set $V_0 = B = 1.0$ which gives a stable soliton solution (see~\cite{Bronski_et_al}).  

\begin{figure}[h]
\begin{center}
\includegraphics[scale=0.5]{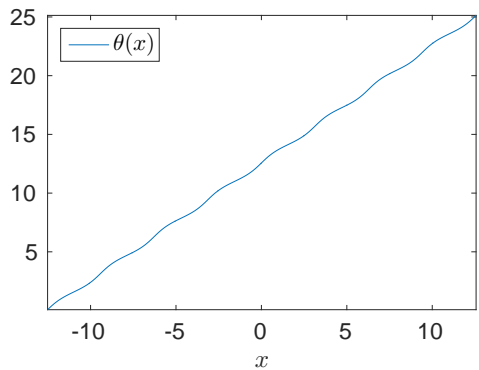}
\end{center}
\caption{The phase function $\theta(x)$ of initial data \eqref{eq:init_nls}.}
\label{fig:theta}
\end{figure}

The equation \eqref{eq:nls_system} can be derived from the energy functional
\begin{equation} \label{eq:energy_nls}
    E(\psi) = \int_{\mathbb{R}} \left[ \frac{1}{4} \abs{\psi_x}^2 + 
    \frac{1}{4} \abs{\psi}^4  - \frac{V_0}{2} \, \sin^2(x) \abs{\psi}^2 \right] \, \dd x .
\end{equation}
As in the example of Subsection~\ref{subsec:derivation}, we first carry out a spatial discretization 
on an equidistant grid with grid size $\Delta x = 8 \pi /n$ and denote by $q_i(t)$ and $p_i(t)$ 
the approximations to $\mathrm{Re} \, \psi(ih,t)$  and $\mathrm{Im} \, \psi(ih,t)$. 
Expressing these approximations as vectors
\begin{equation*}
    \q(t) = \begin{bmatrix}
		  q_1(t) \\ \vdots \\ q_n(t)
              \end{bmatrix},\quad \p(t) = \begin{bmatrix}	p_1(t) \\ \vdots \\ p_n(t) \end{bmatrix},
\end{equation*}
we get the discrete counterpart of the energy functional \eqref{eq:energy_nls},
\begin{equation} \label{eq:Hamiltonian_nls}
\begin{aligned}
    H \big(\p(t),\q(t) \big) &= - \frac{1}{4}\left( \q(t)^T \Delta_n \, \q(t) + \p(t)^T \Delta_n \, \p(t)  
     \right) + \frac{1}{4} \sum\limits_{i=1}^n \left( q_i(t)^2  + p_i(t)^2\right)^2 \\
    & \quad - \frac{V_0}{2} \sum\limits_{i=1}^n \sin^2(x_i) \left( q_i(t)^2 + p_i(t)^2 \right),
    \end{aligned}
\end{equation}
where $ \Delta_n$ is the discretized Laplacian \eqref{eq:laplacian_periodic}.
Setting $\x(t)= \smat{\q(t) \\ \p(t)}$, we get from \eqref{eq:Hamiltonian_nls} 
the Hamiltonian system \eqref{eq:Hamiltonian_system}, i.e.,
\begin{equation} \label{eq:eq1}
\begin{aligned}
  \begin{bmatrix} \q'(t) \\ \p'(t) \end{bmatrix}
  & = - \frac{1}{2} \mat J^{-1} \begin{bmatrix}  \Delta_n & 0 \\ 0 &  \Delta_n \end{bmatrix} \begin{bmatrix} \q(t) \\ \p(t) \end{bmatrix} 
  + \mat J^{-1}  \begin{bmatrix} \big(\q(t) \circ \q(t) + \p(t) \circ \p(t) \big) \circ \q(t) \\
  \big(\q(t) \circ \q(t) + \p(t) \circ \p(t) \big) \circ \p(t) \end{bmatrix} \\
   & \quad - \mat J^{-1} V_0  \begin{bmatrix} \sin^2(\x) \circ \q(t) \\ \sin^2(\x) \circ \p(t) \end{bmatrix},
  \end{aligned}
\end{equation}
where $\circ$ denotes the Hadamard product, $(\u \circ \v)_i = u_iv_i$, and $\sin^2(\x)_i = \sin^2(x_i).$

We set $n=500$ and integrate up to $T=40 \pi$ with step size $h=T/n_t$, where $n_t = 8000$.

The benefits obtained from the symmetry properties of EEMP (Lemma~\ref{lem:eemp_symmetry}) 
are illustrated by Figures~\ref{fig:NLS1} and ~\ref{fig:NLS2} which depict
the relative energy errors and solution errors given by the exponential Euler method and EEMP, when $\U_k \in \mathbb{R}^{n \times 20}$.
The nonsymmetric EE shows a linear growth in energy error and quadratic growth in solution error  whereas
EEMP gives a bounded energy error and a linear growth in solution error.

\begin{figure}[h]
\begin{center}
\includegraphics[scale=0.37]{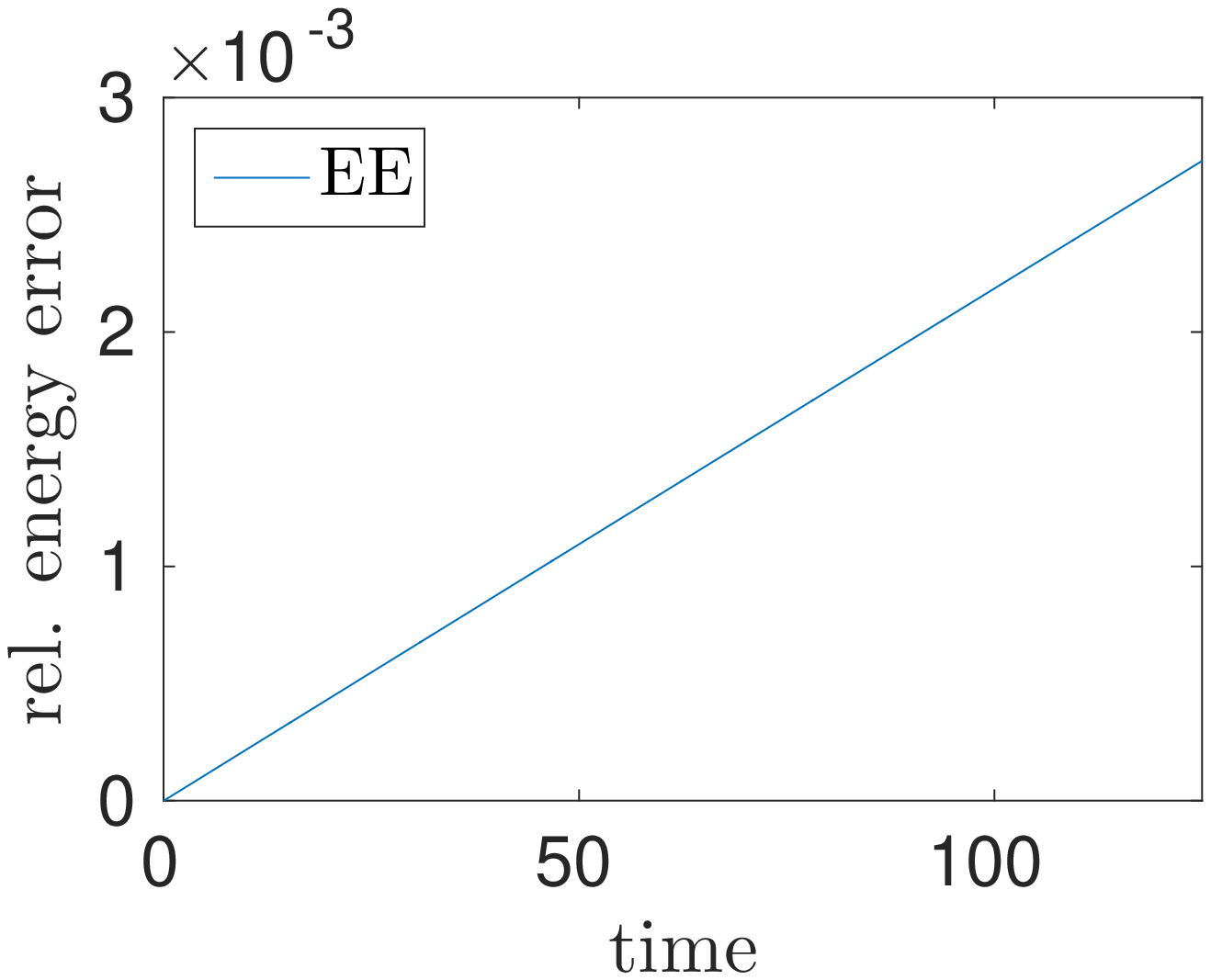}\hspace{-0mm}
\includegraphics[scale=0.37]{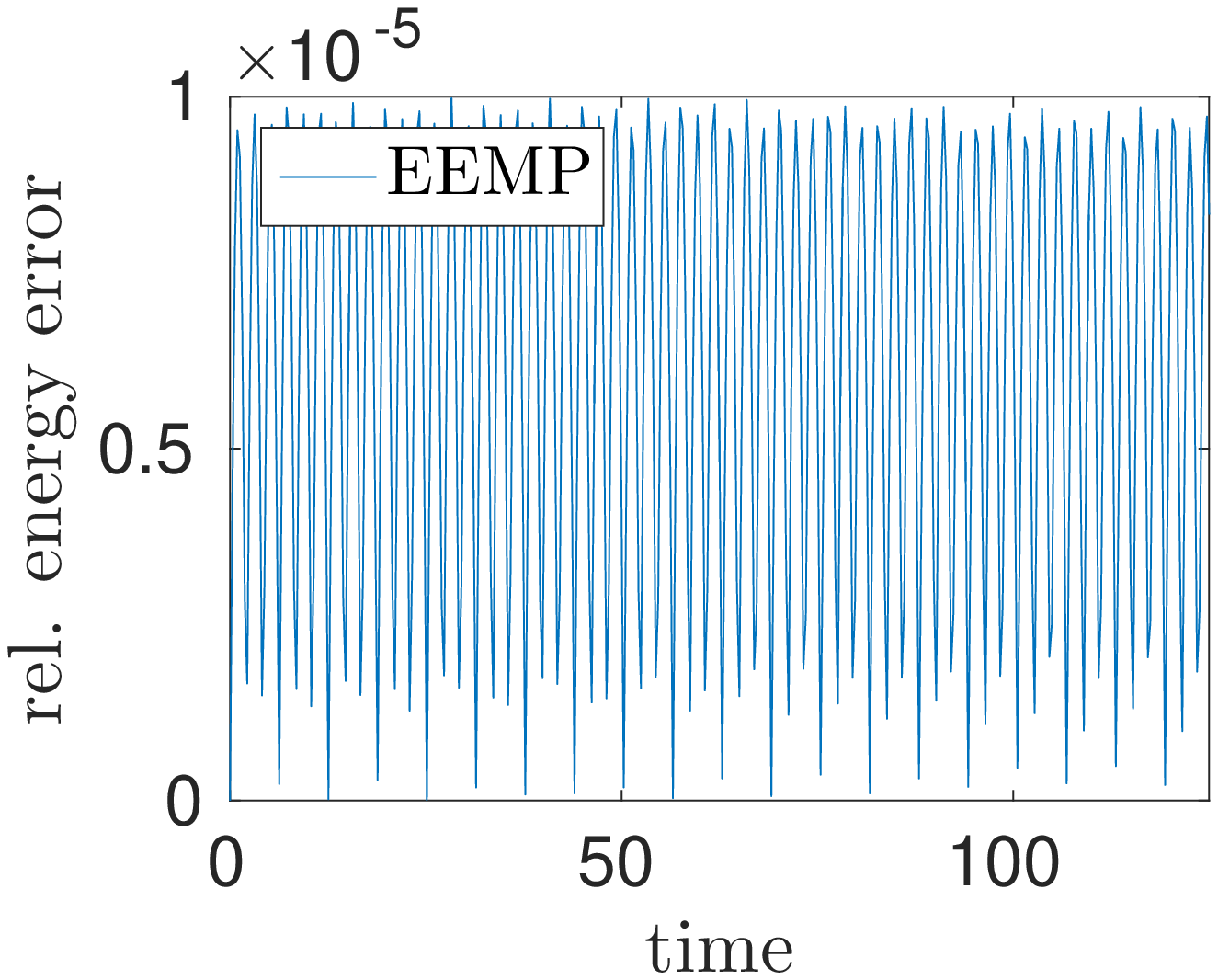}
\end{center}
\caption{NLS and energy errors for the exponential Euler method (EE) and EEMP, when $\U_k$ is produced by the Arnoldi iteration. 
Here $\dim(\U_k) = 20$. }
\label{fig:NLS2}
\end{figure}

\begin{figure}[h]
\begin{center}
\includegraphics[scale=0.5]{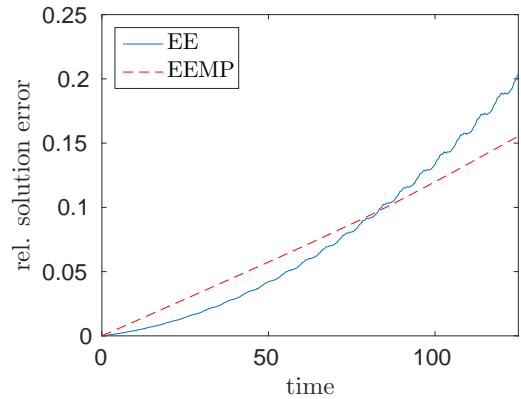}
\end{center}
\caption{NLS and solution errors for EE and EEMP, when $\U_k$ is produced by the Arnoldi iteration. Here $\dim(\U_k) = 20$.}
\label{fig:NLS1}
\end{figure}

Next we set $T=80 \pi$, $n_t = 10000$ and $h=T/n_t$, which implies that the norm of $h \, \H$ is now bigger
and thus larger dimension for $\U_k$ is required.
Differences resulting from the symplecticity of the basis $\U_k$ can be seen in Figure~\ref{fig:NLS_EIMP}
where we compare the IEMP when $\U_k \in \mathbb{R}^{2n \times 24}$ is given by the Arnoldi iteration and the Hamiltonian Lanczos process.
The Arnoldi iteration gives a growth of energy error whereas the Hamiltonian Lanczos iteration shows bounded energy error.

\begin{figure}[h]
\begin{center}
\includegraphics[scale=0.5]{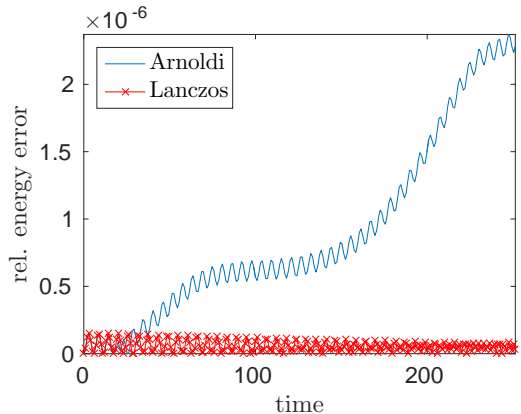}\hspace{-0mm}
\end{center}
\caption{NLS and energy error for IEMP when the basis $\U_k \in \mathbb{R}^{2n \times 24}$ is produced by the
Arnoldi and by the Hamiltonian Lanczos iteration.}
\label{fig:NLS_EIMP}
\end{figure}


\subsection{Nonlinear Klein--Gordon equation}


As a last numerical example we consider the nonlinear Klein--Gordon equation
with periodic boudnary conditions,
\begin{equation*} \label{eq:nls_system}
  \begin{aligned}
      \partial_{tt} u(x,t) &= \partial_{xx} u(x,t) - m^2 u(x,t) - g u(x,t)^3 \\
      u(x,0) &= u_0(x), \quad \quad \textrm{ for all } x\in [0,L]  \\
      u(0,t) &= u(L,t), 
  \end{aligned}
\end{equation*}
The equation is now of the form \eqref{eq:wave_equation} for $f(u) = m^2 u + g u^3$, and after spatial discretization
on the interval $[0,L]$ using finite differences with $n$ points we get a Hamiltonian system with the Hamiltonian \eqref{eq:discrete_Hamiltonian},
where $F(u) = \tfrac{m^2}{2} u^2 + \tfrac{g}{4} u^4$.

Consider the initial data (see~\cite[Example\;1]{Jimenez})
\begin{equation*}
\begin{aligned}
u(x,0) &= A\left(1 + \cos\big(\frac{2 \pi}{L} x \big) \right) \\
u_t(x,0) &= 0.
\end{aligned}
\end{equation*}
Set $A=1$, $m=0.5$, $L=1$, $g=1$. Take $n=400$ discretization points, and consider time integration up to $T=180$
with time step size $h = \frac{T}{n_t}$, $n_t = 9000$.

Here applying EEMP using the basis $\U_k$ generated by the Arnoldi iteration results in an unstable method.
However, the Hamiltonian Lanczos process gives a stable alternative. Figures~\ref{fig:Klein1} show the relative
energy and solution errors for EEMP combined with the Hamiltonian Lanczos process, and for EE combined with the
Arnoldi iteration. For the energy error the symmetric EEMP with symplectic basis $\U_k$ gives a bounded energy error
and also a smaller solution error than EE with orthogonal basis given by the Arnoldi iteration.

\begin{figure}[h!]
\begin{center}
\includegraphics[scale=0.35]{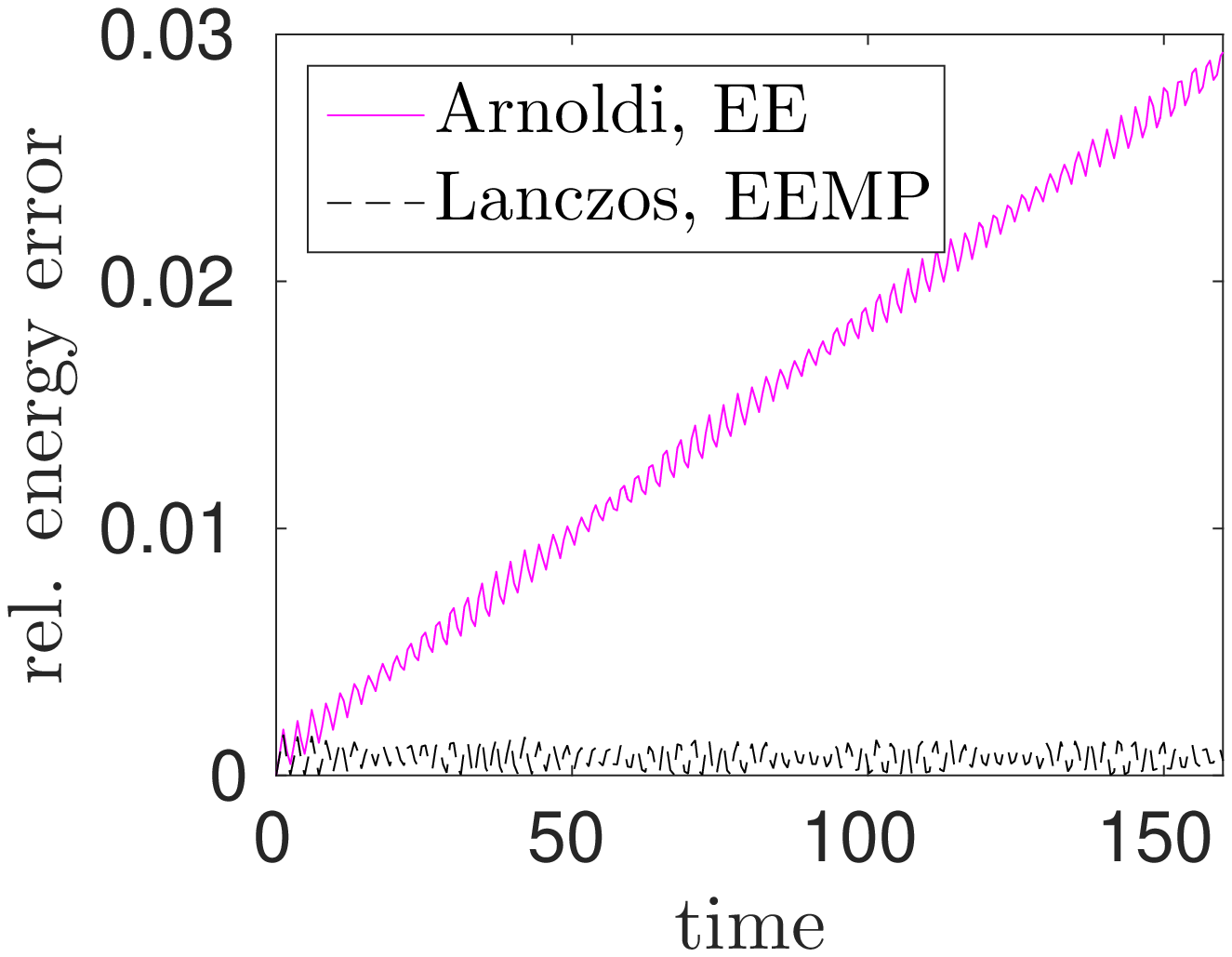}\hspace{-0mm}
\includegraphics[scale=0.35]{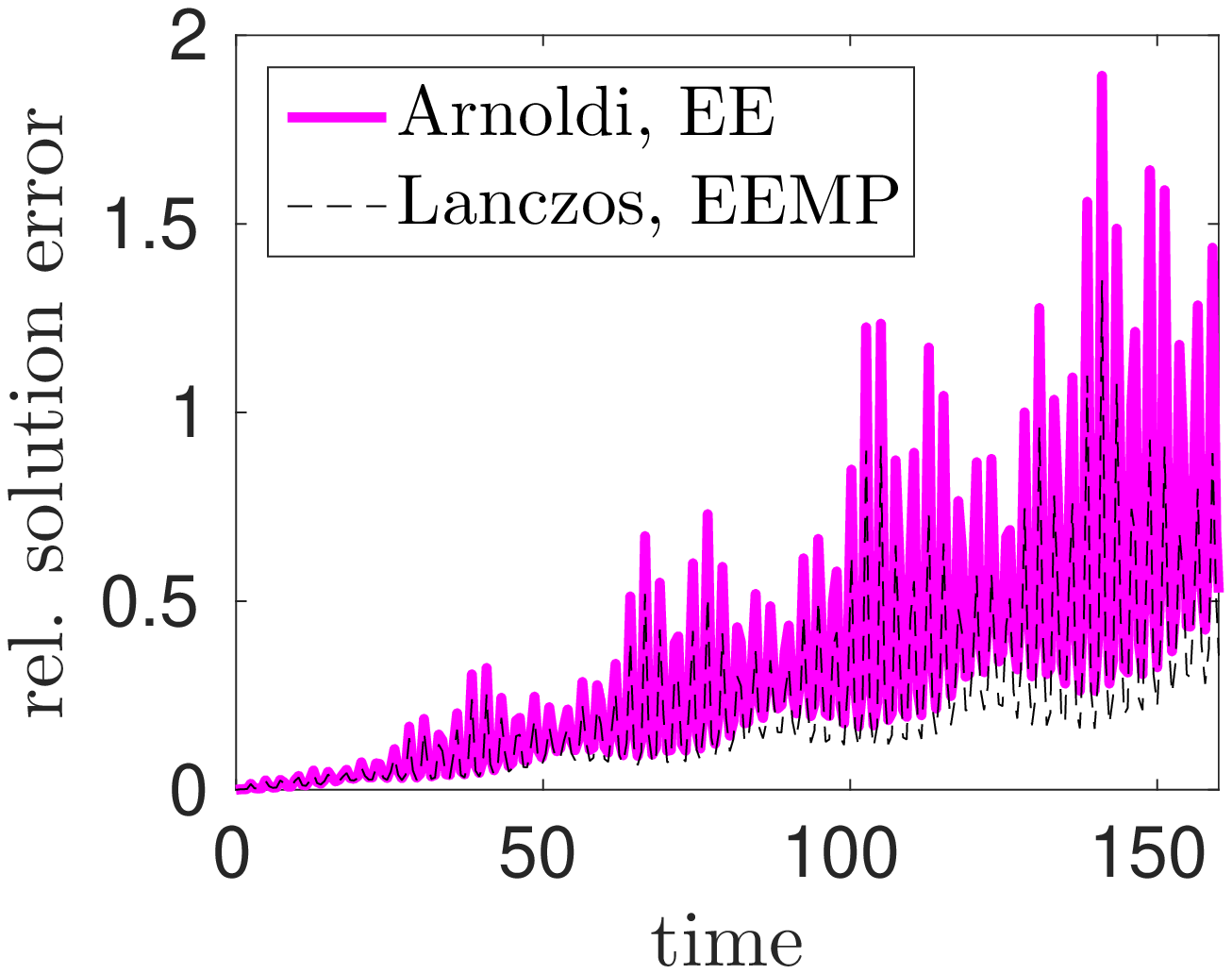}
\end{center}
\caption{Klein-Gordon equation and the energy (left) and the solution (right) errors for EE and EEMP 
when the basis $\U_k \in \mathbb{R}^{2n \times 20}$ is resp. given by the Arnoldi and by the Hamiltonian Lanczos iteration. }
\label{fig:Klein1}
\end{figure}

When applying IEMP, the effect of the symplecticity of $\U_k$ shows up.
Figure~\ref{fig:Klein2} shows the relative energy errors when $\U_k \in \mathbb{R}^{2n \times 22}$
is produced using the Arnoldi iteration and the Hamiltonian Lanczos process.

\begin{figure}[h]
\begin{center}
\includegraphics[scale=0.5]{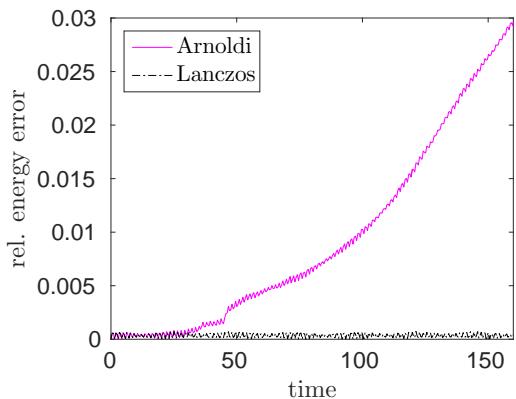}\hspace{-0mm}
\end{center}
\caption{Klein-Gordon equation and the relative energy errors for IEMP,
when the basis $\U_k \in \mathbb{R}^{2n \times 22}$ is 
produced using the Arnoldi iteration and the Hamiltonian Lanczos process.}
\label{fig:Klein2}
\end{figure}

\section{Conclusions and outlook}

The theoretical background of this numerical exploration was the following.
By backward error analysis (see~\cite[Ch.\;IX]{Hairer_Lubich_Wanner}) it can be shown that
applying a symplectic integrator to an integrable Hamiltonian systems gives a
symplectic map $\,\x_j\to\x_{j+1}\,$ and also
\[(\#)\qquad
\begin{cases}
  \ \text{small error in energy uniformly for long time},\\
  \ \text{error growth linear in }t\,
\end{cases}\] 
\null\qquad for exponentially long times (see~\cite[Ch.\;X]{Hairer_Lubich_Wanner}).
Behaviour $\,(\#)\,$ can also be shown for symmetric time integrators when applied to integrable Hamiltonian systems 
(see~\cite[Ch.\;XI]{Hairer_Lubich_Wanner}).

\medskip

Here we have given exponential integration methods which give symplectic maps when applied to
linear Hamiltonian systems.
When using the approximations to nonlinear systems the resulting maps are not symplectic, 
but $\,(\#)\,$ can anyway be observed numerically.
This is escpecially true when using the exponential explicit midpoint rule
\[ \x_{j+1}=\x_j+\U\,e^{h\F}\,\U^\dagger(\x_{j-1}-\x_j)+
2h\,\U\,\phi(h\F)\,\U^\dagger\f(\x_j)\ ,\]
in the symmetric way: the range of $\,\U\,$ contains
\[\x_{j-1}-\x_j, \,\, \f(\x_j), \, \H\f(\x_j), \, \dots, \,\H^{k-1}\f(\x_j)\ .\]
Then $\,\x_{j-1}\,$ is obtained from $\,\x_j\,$ and $\,\x_{j+1}\,$ from the same formula
by changing $\,h\,$ to $\,-h\,$. 
The effect of symplecticity of $\,\U\,$ can be seen numerically when applying the method IEMP (Subsection~\ref{subsec:IEMP})
to nonlinear Hamiltonian problems (see e.g. Figure~\ref{fig:Klein2}).

The numerical experiments clearly show that both preserving the Hamiltonian structure and the time symmetry are important
when applying exponential integrators with Krylov approximations to large scale Hamiltonian systems.
The Hamiltonian Lanczos method appears to be the most efficient method to produce a symplectic basis $\U_k$
among those alternatives that provide the needed Krylov subspace of a given dimension.
However, further study is needed to find an iteration with short $\omega$-orthogonalization recursions
that is more efficient and numerically stable for approximation of the $\phi$ functions.

\end{document}

%% file: AmsLtxMacros.tex
\newcommand{\scs}{\scriptstyle}

\def\svdots{\vbox{\baselineskip=1.5pt\lineskiplimit=0pt
        \kern1.5pt \hbox{$\scs .$}\hbox{$\scs .$}\hbox{$\scs .$}}}

\def\sddots{\mathinner{\raise3pt\vbox{\hbox{$\scs .$}}
    \raise1.5pt\hbox{$\scs .$}\hbox{$\scs .$}}}

\newcommand{\R}{{\mathbb R}}
\newcommand{\C}{{\mathbb C}}

\newcommand{\Cn}{{\C^n}}

\newcommand{\Cnxn}{{\C^{n\times n}}}

\newcommand{\sgn}{\operatorname{sgn}}

\newcommand{\Span}{\operatorname{span}}

\newcommand{\half}{{\frac{1}{2}}}
\newcommand{\thalf}{\tfrac{1}{2}}
\newcommand{\itext}[1]{{\qquad\text{#1}\qquad}}
\newcommand{\abs}[1]{\left| #1 \right|}

\newcommand{\norm}[1]{\left\Vert #1 \right\Vert}

\newcommand{\inprod}[1]{\left\langle #1 \right\rangle}

\newcommand{\sarr}[1]{ \begin{smallmatrix} #1 \end{smallmatrix}}
\newcommand{\smat}[1]{ \left[\sarr{#1}\right]}

\newcommand{\bmat}[1]{ \begin{bmatrix}#1\end{bmatrix}}

\let\tilde\widetilde
\let\hat\widehat

\let\bs\boldsymbol

\def\bphi{{\bs{\varphi}}}
\def\bpsi{{\bs{\psi}}}
\def\bxi{{\bs{\xi}}}
\def\bzeta{{\bs{\zeta}}}

\def\mathbi#1{\textbf{\em #1}}

\renewcommand{\vec}[1]{{\mathbi{#1}}}
\newcommand{\mat}[1]{{\mathbi{#1}}}
 
\def\bc{\vec{c}} 
\def\e{\vec{e}} 
\def\f{\vec{f}}

\def\p{\vec{p}} 
\def\q{\vec{q}} 
\def\r{\vec{r}}

\def\u{\vec{u}} 
\def\v{\vec{v}}
\def\w{\vec{w}} 
\def\x{\vec{x}}
\def\y{\vec{y}} 
\def\z{\vec{z}}

\def\A{\mat{A}} 
 
\def\D{\mat{D}} 
 
\def\F{\mat{F}} 
 
\def\H{\mat{H}} 
\def\I{\mat{I}} 
\def\J{\mat{J}}

\def\M{\mat {M}}
 
\def\Q{\mat {Q}} 

\def\S{\mat {S}} 
\def\U{\mat {U}} 
\def\V{\mat {V}} 
\def\W{\mat {W}}